\documentclass[12pt]{article}
\usepackage{mathrsfs}
\usepackage{amsmath}
\usepackage{float}
\usepackage{cite}
\usepackage{esint}
\usepackage[all]{xy}
\usepackage{amssymb}
\usepackage{amsfonts}
\usepackage{amsthm}
\usepackage{color}
\usepackage{graphicx}
\usepackage{hyperref}
\usepackage{bm}
\usepackage{indentfirst}
\usepackage{geometry}
\theoremstyle{plain}
\newtheorem{thm}{Theorem}[section]

\newtheorem{defn}[thm]{Definition}
\newtheorem{prop}[thm]{Proposition}

\newtheorem{them}[thm]{Theorem}
\theoremstyle{remark}
\newtheorem{ex}[thm]{Example}
\newtheorem{rmk}[thm]{Remark}

\newenvironment{usethmcounterof}[1]{%
  \renewcommand\thethm{\ref{#1}}\thm}{\endthm\addtocounter{thm}{-1}}

\newcommand{\Z}{\mathbb{Z}}

\newcommand{\R}{\mathbb{R}}

\newcommand{\C}{\mathbb{C}}

\newcommand{\CF}{\mathrm{CF}}
\newcommand{\im}{\mathrm{Im}}


\title{Uniqueness of holomorphic quilts lifted from holomorphic bigons on surfaces}
\author{Zuyi Zhang\thanks{Beijing International Center for Mathematical Research, Beijing University, China. \emph{Email: zhangzuyi1993@hotmail.com}}\thanks{MSC class: 53D37 53D40 32Q65 57K41}}

\date{ }

\begin{document}

\maketitle

\begin{abstract}
   In the previous paper \cite{zhang2024construction}, the author constructed holomorphic quilts from the bigons of the Lagrangian Floer chain group after performing Lagrangian composition. This paper proves the uniqueness of such holomorphic quilts. As a consequence, it provides a combinatorial method for computing the boundary map of immersed Lagrangian Floer chain groups when the symplectic manifolds are closed surfaces. One outcome is the construction of many examples exhibiting figure-eight bubbling, which also confirms a conjecture of Cazassus–Herald–Kirk–Kotelskiy \cite{cazassus2020correspondence}.
\end{abstract}

\section{Introduction}
Lagrangian Floer homology was first defined by Floer, based on Gromov’s theory of pseudo-holomorphic curves \cite{floer1988morse}, to solve Arnold’s conjecture. When the symplectic manifold is a closed surface, Lagrangian Floer theory becomes combinatorial and has been studied extensively by Abouzaid \cite{abouzaid2008fukaya}. A natural question arises: when the symplectic manifold is a Cartesian product of two closed surfaces, can the Lagrangian Floer theory still be computed combinatorially? By combining the results of this paper with those of \cite{zhang2024construction}, we provide a method to do so in the case where one of the Lagrangians is a Cartesian product of closed curves from the two surfaces.

In \cite{zhang2024construction}, the author constructed holomorphic quilts from the bigons of the Lagrangian Floer chain group after performing Lagrangian composition. This paper establishes the uniqueness of such holomorphic quilts. The main theorem is as follows.

\begin{them}\label{thm:main}
Let $(F_1.\omega_1)$ and $(F_2,\omega_2)$ be two symplectic closed surfaces equipped with compatible almost complex structures. Suppose $L_i\looparrowright F_i$ $\mathrm{(}i=1,2\mathrm{)}$ and $F\looparrowright (F_1\times F_2,\omega_1\times(-\omega_2))$ are Lagrangian immersions such that $L_1\times L_2$ intersects $F$ transversely in $F_1\times F_2$. Furthermore, assume $F$ and $L_2$ are composable. Assume that
    \[
    u(x,y):\R\times[0,1]\rightarrow F_2
    \]
    is a holomorphic map such that     
    \begin{itemize}
        \item $\lim_{x\rightarrow\pm\infty}u(x,y)=x_{\pm}$, where $x_\pm$ are two points in the generalized intersection of $L_1\circ F$ and $L_2$,
        \item $u$ has its boundary in $L_1\circ F$ and $L_2$ $($Definition \ref{def:lagbd}$)$,
        \item the image of $u$ is contained in the image of $F\looparrowright F_1\times F_2$ after projecting to the first factor.
    \end{itemize} 
    Then there is a unique holomorphic map $\tilde u(x,y):\R\times[0,1]\rightarrow F_1\times F_2$ such that
    \begin{itemize}
        \item $\lim_{x\rightarrow\pm\infty}\tilde u(x,y)=x_{\pm}$, where $\tilde x_\pm$ are two points in the generalized intersection of $L_1$ and $F\circ L_2$ corresponding to $x_\pm$ by using the canonical identification in Proposition \ref{prop:sg},
        \item $\tilde u$ has its boundary in $L_1\times L_2$ and $F$,
        \item the image of $\tilde{u}|_{\R\times\{0\}}$ in the $L_2$ component is the same as $\im(u|_{\R\times\{1\}})$ in the sense of set. 
    \end{itemize}
\end{them}

The idea of the proof is as follows. Unlike the argument in \cite{zhang2024construction}, where the method of Wehrheim–Woodward \cite{wehrheim2009floer} applies directly, the situation in this paper involves figure-eight bubbling, which the Wehrheim–Woodward approach explicitly excludes. Their argument requires the absence of figure-eight bubbling, but such bubbling necessarily occurs in our setting. Suppose on the contrary, there are two holomorphic maps $(u_1,u_2)$ and $(w_1,w_2)$ satisfy the conclusion of the main theorem. Then $u_2$ and $w_2$ are holomorphic discs sharing the $L_2$ side with $u$. During the strip-shrinking process, one can remove the keyhole neighborhoods (Figure \ref{fig:keyholwremove}) around the points that contributes to the bubblings from the domain disc and the Riemann mapping theorem guarantees uniqueness.

One of the motivations for this work comes from a conjecture of Bottman-Wehrheim \cite{bottman2018gromov}. This conjecture states that by adding bounding cochains corresponding to the figure-eight bubblings (c.f. \cite{bottman2018gromov}) to twist the boundary maps of the immersed Lagrangian Floer chain group and the immersed quilted Lagrangian Floer chain group, these two chain groups become chain complexes and are isomorphic. 

Quilted Lagrangian Floer theory was introduced by Wehrheim-Woodward \cite{wehrheim2010quilted}. In Wehrheim-Woodward \cite{wehrheim2009floer}, they showed that, under certain additional assumptions (e.g. monotonicity), these two theories are isomorphic when all Lagrangians are embedded. Fukaya later proved the isomorphism indicated by Bottman and Wehrheim's conjecture in an algebraic manner \cite{fukaya2017unobstructed}. However, how this isomorphism relates to strip shrinking and the associated figure-eight bubbling remains unclear. Our goal is to understand this relationship.

Another motivation concerns the Atiyah–Floer conjecture, which connects Lagrangian Floer theory with instanton Floer theory. On the Lagrangian Floer side, the relevant symplectic manifold (possibly with singularities) is given by the space of traceless $\mathrm{SU}(2)$ representations of the fundamental group of a closed surface $F$, modulo conjugations. Goldman \cite{goldman1984symplectic} demonstrated that this space can be equipped with a symplectic structure. The Lagrangians are defined as the restrictions of the traceless $\mathrm{SU}(2)$ representations of the fundamental group (modulo conjugations) of a 3-manifold to the boundary, where the boundary of this 3-manifold is $F$. Here, the Lagrangians are typically immersions into the traceless $\mathrm{SU}(2)$ representations of $\pi_1(F)$ modulo conjugations. To understand the Lagrangian Floer side of the Atiyah-Floer conjecture, Lagrangian compositions and the quilted Lagrangian Floer theory is necessary. Examples can be found in K. Smith \cite{smith2023perturbed}, Cazassus-Herald-Kirk-Kotelskiy \cite{cazassus2020correspondence}, as well as Herald-Kirk \cite{herald2024endomorphism}. In Cazassus-Herald-Kirk-Kotelskiy \cite{cazassus2020correspondence}, a conjecture about the figure-eight bubblings was proposed (Section 10.4 in \cite{cazassus2020correspondence}). This paper gives a positive answer to this conjecture.

The paper is organized as follows. In Section \ref{sec:2}, we introduce the definitions of immersed Lagrangian Floer theory and quilted Lagrangian Floer theory. Section \ref{sec:3} discusses the strip-shrinking and examples with figure-eight bubblings. Section \ref{sec:4} contains the proof of the main theorem. We apply the main theorem to calculate some examples of the immersed quilted Lagrangian Floer theory as well as to prove a conjecture of Cazassus-Herald-Kirk-Kotelskiy \cite{cazassus2020correspondence} in Section \ref{sec:5}.\\

\noindent\textbf{Acknowledgments:} The author would like to thank Paul Kirk, Chris Woodward, and Nate Bottman for helpful discussions.

\section{Preliminary}\label{sec:2}
The aim of this section is to introduce the necessary background of immersed Lagrangian Floer theory and immersed quilted Lagragnian-Floer theory.\\

\begin{defn}
Suppose that $l:L\looparrowright X$ is an immersion from an $n$ dimensional manifold $L$ to a 2n dimensional symplectic manifold $(X,\omega)$. The immersion $l:L\looparrowright X$ is called a {\bf Lagrangian immersion} if
\[
l^*\omega=0.
\]
\end{defn}

\begin{ex}\label{exp:2.1}
Any immersed curve $C$ in a symplectic surface $(F,\omega)$ is Lagrangian. Since $\omega$ is skew-symmetric, then $\omega(v_x,v_x)=0$ for all $x\in C$ and $v_x\in T_xC$.
\end{ex}


\begin{defn}\label{def:lagint}
Let $l_1:L_1\looparrowright X$ and $l_2:L_2\looparrowright X$ be two Lagrangian immersions into a symplectic manifold $X$. The {\bf generalized intersection} $L_1\times_X L_2$ is defined as the fiber product
\[
L_1\times_X L_2=\{(x_1,x_2)\in L_1\times L_2|\ l_1(x_1)=l_2(x_2)\}.
\]
\end{defn}

\begin{rmk}
Since $l_1:L_1\looparrowright X$ and $l_2:L_2\looparrowright X$ are immersions, $l_i^{-1}(x)$ may contain more than one element for $x\in l_1(L_1)\cap l_2(L_2)$ for $i=1,2$. It is necessary to distinguish these preimages. So we use the term ``generalized intersection" here.
\end{rmk}




\begin{defn}
Let $(X,\omega)$ be a symplectic manifold. An {\bf almost complex structure} is a bundle map $J:TX\rightarrow TX$ such that 
\[
J^2=-id.
\]
An almost complex structure $J$ is called {\bf $\mathbf\omega$-compatible} if
\begin{itemize}
    \item $\omega(\cdot,J\cdot)$ is a Riemannian metric,
    \item $\omega(J\cdot,J\cdot)=\omega(\cdot,\cdot)$.
\end{itemize}

\end{defn}

Let $(X,\omega)$ be a symplectic manifold. Then the space of $\omega$-compatible almost complex structures is nonempty and contractible \cite{mcduff2017introduction}.

\begin{defn}
Let $l_1:L_1\looparrowright X$ and $l_2:L_2\looparrowright X$ be two Lagrangian immersions into a symplectic manifold $X$. Let $(x_1,x_2)\in L_1\times_XL_2$ be an generalized intersection point of these two Lagrangian immersions. We say {\bf $\mathbf{L_1}$ intersects $\mathbf{L_2}$ transversely at $\mathbf{(x_1,x_2)}$} if 
\[
\mathrm{Im}(df^1)_{x_1}\cap\mathrm{Im}(df^2)_{x_2}=\{0\}.
\]
If the above equation holds for every generalized intersection point, then {\bf $\mathbf{L_1}$ intersects $\mathbf{L_2}$ transversely}.
\end{defn}

\begin{defn}
Let $X$ be a symplectic manifold with compatible almost complex structure $J$. The complex structure $j$ over the strip $\mathbb R\times [0,1]$ is induced from the standard complex structure over $\C$. Then $u:\mathbb R\times [0,1]\rightarrow X$ is called a {\bf J-holomophic map} if
\[
\bar\partial_Ju:=\frac{1}{2}(du+J\circ du\circ j)=0.
\]
\end{defn}

\begin{defn}\label{def:lagbd}
Let $u$ be a smooth map from the strip $\mathbb R\times [0,1]$ to a symplectic manifold $X$ with a compatible almost complex structure. Assume $L_i\looparrowright X$ are 2 immersed Lagrangian submanifolds, for $i=1,2$. We say that {\bf $\mathbf u$ has its boundary in $\mathbf L_i$} if there are lifts $\tilde u_i$ such that the following diagram commutes:
\begin{equation*}
    \xymatrix{
    &L_i\ar[d]\\
    \mathbb R\times \{i-1\}\ar[r]^{\ \ \ u}\ar@{-->}[ur]^{\tilde u_i}&X,
    }
\end{equation*}
for $i=1,2$.
\end{defn}

\begin{defn}
    Let $l_1:L_1\looparrowright X$ and $l_2:L_2\looparrowright X$ be two Lagrangian immersions into a symplectic manifold $X$ such that $L_1$ intersects $L_2$ transversely. Let $J$ be a $\omega$-compatible almost complex structure on $X$. A \textbf{holomorphic bigon connecting two elements} $\mathbf{x_\pm}$ is a $J$-homorphic map $u:\R\times[0,1]\rightarrow X$ satisfying the following conditions:
    \begin{itemize}
        \item $u$ has its boundary in $L_i$, $i=1,2$,
        \item $\lim_{s\rightarrow\pm\infty}u(s,\cdot)=x_\pm$,
        \item the Fredholm index of the Cauchy Riemann operator is 1.
    \end{itemize}
\end{defn}


\begin{defn}
Let $L_1\looparrowright X$ and $L_2\looparrowright X$ be two Lagrangian immersions into a symplectic manifold $X$. Assume that $L_1$ and $L_2$ intersect transversely. The {\bf Lagrangian Floer chain group} $\CF(L_1,L_2;X)$ is defined as the $\Z_2$ vector space generated by points in $L_1\times_XL_2$.
\end{defn}

Next we define the boundary maps of Lagrangian Floer chain groups. Let $X$ be a symplectic manifold with compatible almost complex structure $J$. Let $L_1\looparrowright X$ and $L_2\looparrowright X$ be two Lagrangian immersions. For two generalized intersection points $x_+,x_-\in L_1\times_X L_2$, denote $\mathcal{X}(x_+,x_-)$ as the set of $J$-holomorphic maps $u:\mathbb R\times [0,1]\rightarrow X$ connecting $x_\pm$ with the following properties:
\begin{itemize}
    \item $\bar\partial_Ju=0$,
    \item $\int u^*\omega<+\infty$,
    \item $u$ has its boundary in $L_i$, $i=1,2$,
    \item $\lim_{t\rightarrow\pm\infty}u(t,x)=x_\pm$.
\end{itemize}

There is a translation $\R$ action on $\mathcal{X}(x_+,x_-)$:
\begin{align*}
    \mathcal{X}(x_+,x_-)\times\R&\rightarrow\mathcal{X}(x_+,x_-)\\
    (u(\cdot,\cdot),t)\ \ &\mapsto u(\cdot+t,\cdot).
\end{align*}
The quotient space of $\mathcal{X}(x_+,x_-)$ under this $\R$ action is the moduli space connecting $x_+$ and $x_-$:
\[
\mathcal{M}(x_+,x_-)=\mathcal{X}(x_+,x_-)/\mathbb R.
\]
If $\dim\mathcal{M}(x_+,x_-)=0$, then for generic $J$-holomorphic structures on $X$, the moduli space $\mathcal{M}(x_+,x_-)$ is a compact, smooth manifold \cite{zhang2024construction}. Consequently, $\mathcal{M}(x_+,x_-)$ contains finitely many elements when $\dim\mathcal{M}(x_+,x_-)=0$. The 
\textbf{boundary map} $\mu^1$ of the Lagrangian Floer complex $\CF(L_1,L_2;X)$ is defined as 
\begin{align*}
\mu^1:\CF(L_1,L_2;X)&\rightarrow \ \ \ \ \ \ \ \ \ \ \ \ \CF(L_1,L_2;X)\\
x_+\ \ \ \ \ \ &\mapsto\sum_{\substack{x_-\in L_1\cap L_2\\ \dim\mathcal{M}(x_+,x_-)=0}}\#_{Mod\, 2}\mathcal{M}(x_+,x_-)x_-,
\end{align*}
where $\#_{Mod\, 2}\mathcal{M}(x_+,x_-)$ is the number of the elements in $\mathcal{M}(x_+,x_-)$ modulo $2$.\\

When the symplectic manifold is two-dimensional, the number of $J$-holomorphic maps connecting two generalized intersection points $x_+$ and $x_+$ of Lagrangian immersions corresponds to the number of discs connecting these two points $x_\pm$,  according to the
Riemann mapping theorem.\\

\textbf{The following provides the definitions of Lagrangian correspondences and Lagrangian compositions}. The concept of Lagrangian correspondence is essential for defining Lagrangian compositions. Lagrangian correspondences are specific types of Lagrangian immersions, as defined below. When both terms— Lagrangian correspondence and Lagrangian immersion— are applicable, we use “Lagrangian correspondence” to emphasize the composition aspect of Lagrangian immersions.

\begin{defn}
Given symplectic manifolds $(X_0, \omega_0)$ and $(X_1, \omega_1)$, a {\bf Lagrangian correspondence} from $X_0$ to $X_1$ is a Lagrangian immersion
\[
l=(l_0,l_1):L\rightarrow X_0
\times X_1,
\]
where $X_0\times X_1$ is equipped with the symplectic structure $\omega_0\times(-\omega_1)$.
\end{defn}

\begin{defn}\label{def:lacom}
Let $(X_0, \omega_0)$, $(X_1, \omega_1)$, $(X_2, \omega_2)$ be three symplectic manifolds. Then $(X_0\times X_1,\omega_0\times (-\omega_1))$, $(X_1\times X_2,\omega_1\times (-\omega_2))$, $(X_0\times X_2,\omega_0\times (-\omega_2))$ are symplectic manifolds. Assume that 
\[
l^0=(l_0^0,l^0_1):L_{01}\rightarrow X_0
\times X_1\quad and\quad l^1=(l_1^1,l^1_2):L_{12}\rightarrow X_1
\times X_2
\]
are Lagrangian correspondences. The {\bf Lagrangian composition} $L_{01}\circ L_{12}$ of $L_{01}$ and $L_{12}$ is defined as
\begin{equation}\label{equ:lagc}
    l^0\circ l^1:=(l^0_0,l^1_2):L_{01}\circ L_{12}\rightarrow X_0\times X_2,
\end{equation}
where 
\[
L_{01}\circ L_{12}:=L_{01}\times_{X_1} L_{12}=\{(x,y)\in L_{01}\times L_{12}|l^0_1(x)=l^1_1(y)\}.
\]
\end{defn}

The composition $l^0\circ l^1:L_{01}\circ L_{12}\looparrowright X_0\times X_1$ in Definition \ref{def:lacom} is not always a Lagrangian immersion. The following proposition gives a sufficient criterion.

\begin{prop}[\cite{wehrheim2010functoriality}]\label{prop:lagco}
Let $(X_0, \omega_0)$, $(X_1, \omega_1)$ and $(X_2, \omega_2)$ be three symplectic manifolds. Then $(X_0\times X_1,\omega_0\times (-\omega_1))$, $(X_1\times X_2,\omega_1\times (-\omega_2))$ and $(X_0\times X_2,\omega_0\times (-\omega_2))$ are symplectic manifolds. Suppose that 
\[
\Delta_{X_1}=\{(x,x)|x\in X_1\}\subset X_1\times X_1
\]
is the diagonal and 
\[
l^0=(l_0^0,l^0_1):L_{01}\rightarrow X_0
\times X_1,l^1=(l_1^1,l^1_2):L_{12}\rightarrow X_1
\times X_2
\]
are Lagrangian correspondences. If $l^0\times l^1$ is transverse to $X_0\times\Delta_{X_1}\times X_2$, then 
\begin{itemize}
    \item $L_{01}\circ L_{12}$ is a smooth manifold,
    \item $l^0\circ l^1:L_{01}\circ L_{12}\looparrowright X_0\times X_2$ is a Lagrangian immersion.
\end{itemize}
\end{prop}

\begin{defn}\label{def:composable}
Given Lagrangian correspondences $L_{01}$ and $L_{12}$ in $(X_0
\times X_1,\omega_0\times (-\omega_1))$ and $(X_1
\times X_2,\omega_1\times (-\omega_2))$, resp. These two Lagrangian correspondences $L_{01}$ and $L_{12}$ are said to be {\bf composable} if all the assumptions in Proposition \ref{prop:lagco} are satisfied.
\end{defn}

If $L_1\looparrowright X_1$ is a Lagrangian immersion, then $L_1$ can be regarded as a Lagrangian correspondence from $\{pt\}$ to $X_1$. Let $L_{12}\looparrowright X_1\times X_2$ be a Lagrangian correspondence. If $L_1$ and $L_{12}$ are composable, then $L_1\circ L_{12}\looparrowright X_2$ is a Lagrangian immersion. Similarly, if $L_2$ is a Lagrangian immersion into $X_2$, then $L_2$ can be viewed as a Lagrangian correspondence from $X_2$ to $\{pt\}$. If $L_2$ is composable with $L_{12}$, then $L_{12}\circ L_2\looparrowright X_1$ is a Lagrangian immersion.\\

\textbf{We conclude this section by introducing the quilted Lagrangian Floer theory}. From now on, the symplectic manifolds are assumed to be closed surfaces. Let $g_1: F\rightarrow F_1$ and $g_2:F\rightarrow F_2$ be two maps between closed surfaces. Denote $\omega_1$, $\omega_2$, $\omega_1\times(-\omega_2)$ as the symplectic forms on $F_1$, $F_2$, $F_1\times F_2$ respectively. Assume that
\[
(g_1,g_2):F\rightarrow F_1\times F_2
\] 
is a Lagrangian correspondence from $F_1$ to $F_2$. Suppose that $L_1\looparrowright F_1$ and $L_2\looparrowright F_2$ are two immersed curves, composable with $F$. Define the {\bf quilted Lagrangian Floer chain group} $\CF^Q(L_1,F,L_2;F_1,F_2)$ as the abelian group generated by elements in $F\times_{F_1\times F_2}(L_1\times L_2)$ over $\mathbb Z_2$.\\
Let $J_1,J_2,J_1\times (-J_2)$ be the almost complex structures compatible with $\omega_1,\omega_2,\omega_1\times(-\omega_2)$, respectively. Suppose that 
\[
u_i:\R\times[i-1,i]\rightarrow F_i,\ \ i=1,2
\]
are $J$-holomorphic maps such that the following diagram commute
\begin{equation*}
    \xymatrix{
    &L_1\ar[d]\\
    \mathbb R\times \{0\}\ar[r]^{\ \ \ u_1}\ar@{-->}[ur]^{\exists\tilde u_1}&F_1,
    }\ \ 
    \xymatrix{
    &L_2\ar[d]\\
    \mathbb R\times \{2\}\ar[r]^{\ \ \ u_2}\ar@{-->}[ur]^{\exists\tilde u_2}&F_2,
    }\ \ 
    \xymatrix{
    &F\ar[d]\\
    \mathbb R\times \{1\}\ar[r]^{ (u_1,u_2)}\ar@{-->}[ur]^{\exists(\Bar u_1,\Bar u_2)}&F_1\times F_2.
    }
\end{equation*}
If moreover 
\[
\lim_{x\rightarrow\pm\infty}(u_1,u_2)(x,y)=(x_\pm,y_\pm),\ \ x_\pm\in F\times_{F_1\times F_2}(L_1\times L_2),
\]
then $(u_1,u_2)$ is called a {\bf quilted holomorphic disc connecting} $\mathbf{(x_\pm,y_\pm)}$. For fixed $(x_\pm,y_\pm)\in F\times_{F_1\times F_2}(L_1\times L_2)$, the space of all such pairs $(u_1,u_2)$ connecting $(x_\pm,y_\pm)$ is denoted as $\mathcal{X}^Q((x_+,y_+),(x_-,y_-))$. The moduli space of quilted holomorphic strips connecting $x_\pm$ is defined as
\[
\mathcal{M}^Q((x_+,y_+),(x_-,y_-)):=\mathcal{X}^Q((x_+,y_+),(x_-,y_-))/\R.
\]
According to \cite{zhang2024construction}, if the expected dimension of $\mathcal{M}^Q((x_+,y_+),(x_-,y_-))$ is 0, then $\mathcal{M}^Q((x_+,y_+),(x_-,y_-))$ is a compact smooth manifold for generic $J_1$ and $J_2$. The {\bf boundary map} $\mu^1_Q$ of the quilted Lagrangian Floer complex $\CF^{Q}(L_1,F,L_2;F_1,F_2)$ is defined as
\begin{equation*}
    \begin{aligned}
    \mu^1_Q&:\CF^{Q}(L_1,F,L_2;F_1,F_2)\rightarrow\CF^{Q}(L_1,F,L_2;F_1,F_2)\\
    (x_+,y_+)&\mapsto\sum_{\substack{x_-\in F\times_{F_1\times F_2}(L_1\times L_2)\\ \dim\mathcal{M}^Q((x_+,y_+),(x_-,y_-))=0}}\#_{Mod\, 2}\mathcal{M}^Q((x_+,y_+),(x_-,y_-))(x_-,y_-).
    \end{aligned}
\end{equation*}

\begin{prop}\label{prop:sg}
Let $L_1$, $F$, $L_2$ be Lagrangian immersions in $(F_1,\omega_1)$, $(F_1\times F_2,\omega_1\times(-\omega_2))$, $(F_2,\omega_2)$, resp. Suppose $L_1$ and $F$, $F$ and $L_2$ are composable. Then the generators of the complexes $\CF(L_1,L_{12}\circ L_2; F_1)$, $\CF(L_1\circ L_{12},L_2;F_2)$, $\CF^{Q}(L_1,F,L_2;F_1,F_2)$ can be identified canonically.
\end{prop}

\begin{proof}
A direct calculation using definitions guarantees this proposition. Details can be found in \cite{zhang2024singularities} Proposition 8.6.
\end{proof}

\section{Figure-eight Bubblings and Strip-shrinking}\label{sec:3}
In this section, the notation for figure-eight bubblings as well as the strip-shrinking process is explained. Figure-eight bubblings was introduced by Bottman-Wehrheim 
\cite{bottman2018gromov}. They showed that the figure-eight bubblings cannot be avoided for generic almost complex structures when applying the strip-shrinking \cite{bottman2018gromov} argument to transition from holomorphic quilts to holomorphic discs. At the end of this section, we prove that figure-eight bubblings necessarily occur when performing the strip-shrinking near bifold points (Definition \ref{def:bifold}).

\begin{defn}[Bottman-Wehrheim \cite{bottman2018gromov}]
    Let $L_1$, $F$, $L_2$ be Lagrangian immersions in $(F_1,\omega_1)$, $(F_1\times F_2,\omega_1\times(-\omega_2))$, $(F_2,\omega_2)$, resp. \textbf{A figure-eight bubbling bwtween $F$ and $L_2$} $($Figure \ref{fig:fig8}$)$ is a pair of smooth maps $(w_1,w_2)$ such that:
    \begin{itemize}
        \item $w_1:\R\times[-\frac{1}{2},\frac{1}{2}]\rightarrow F_1$ and $w_1:\R\times[\frac{1}{2},+\infty]\rightarrow F_2$ are holomorphic,
        \item $w_1(s,-\frac12)$ and $(w_1(s,\frac12),w_2(s,\frac12))$ can be lifted to $ L_1$ and $F$ respectively.
    \end{itemize}
\end{defn}

\begin{figure}[htbp]
\centering
\begin{minipage}[t]{0.4\textwidth}
\centering
\includegraphics[width=4cm]{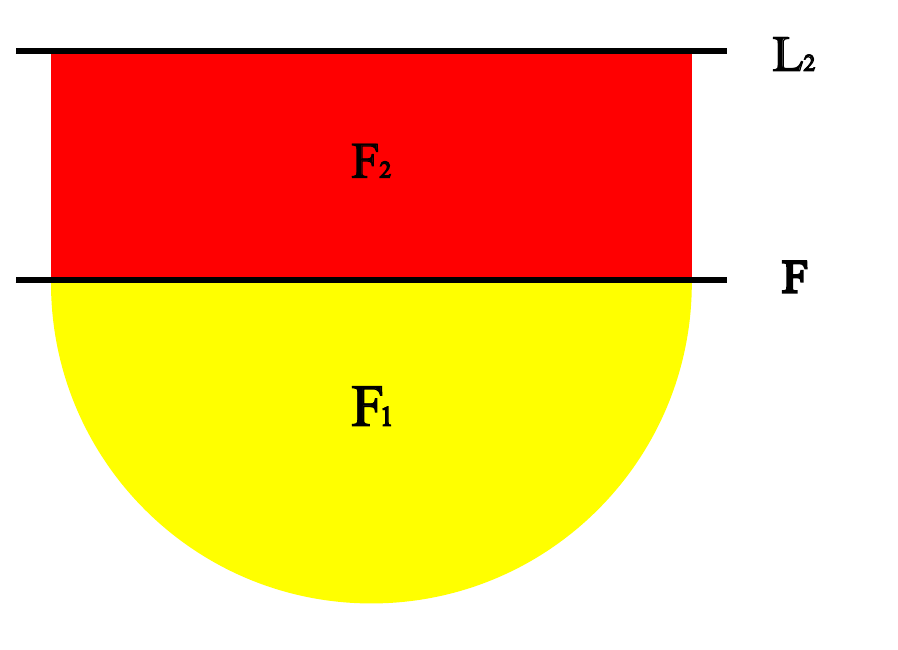}
\caption{Figur-eight bubbling.}
\label{fig:fig8}
\end{minipage}
\begin{minipage}[t]{0.5\textwidth}
\centering
\includegraphics[width=4cm]{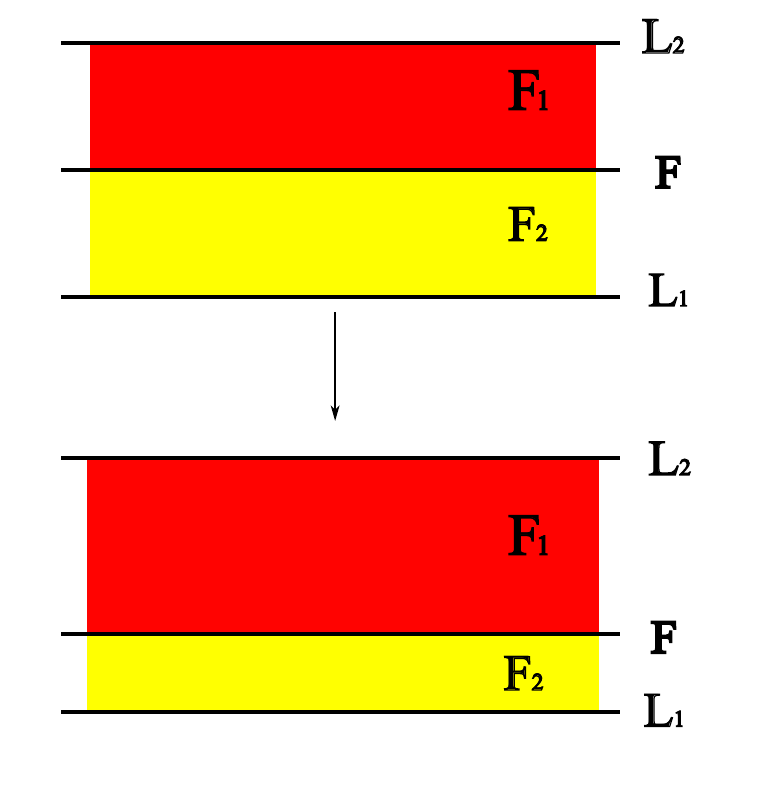}
\caption{Shrinking the lower layer.}
\label{fig:strsh}
\end{minipage}
\end{figure}

\begin{rmk}
    In the notation of Bottman-Wehrheim \cite{bottman2018gromov},
    \[
    M_0=pt,\ \ M_1=F_1,\ \ M_2=F_2.
    \]
\end{rmk}

Given a quilted holomorphic disc $(u_1,u_2)$ connecting two elements in $F\times_{F_1\times F_2}(L_1\times L_2)$, the \textbf{$F_1$-strip-shrinking} (Figure \ref{fig:strsh}) in Bottman-Wehrheim \cite{bottman2018gromov} refers to a family of smooth maps $\{(u^t_1,u^t_2)\}_{t\in[0,1]}$ such that:
\begin{itemize}
    \item $u_1^t:\R\times [\delta_t,1]\rightarrow F_1$ and $u_2^t:\R\times [1,2]\rightarrow F_2$ are holomorphic with $t\in[0,1)$, where $\delta_t\in[0,1]$ is an increasing function with respect to $t$ for $\delta_0=0$ and $\delta_1=1$,
    \item $(u^t_1,u^t_2)=(u_1,u_2)$ when $t=0$,
    \item $((u^t_1(s,\delta_t),u^t_2(s,1))$ and $((u^t_1(s,1),u^t_2(s,1))$ can be lifted to $L_1\times L_2$ and $F$, respectively.
\end{itemize}
The main theorem in Bottman-Wehrheim \cite{bottman2018gromov} states that during the strip-shrinking process, if the first derivative of $(u_1^t,u_2^t)$ becomes unbounded as $t\rightarrow1$,  then figure-eight bubblings occur between $F$ and $L_2$ for generic complex structures. The term ``complex structures" is presented here in place of ``almost complex structures" as in \cite{bottman2018gromov}, since all almost complex structures on closed surfaces are integrable.\\

Some preparations are needed before presenting an example with figure-eight bubblings.
\begin{defn}[\cite{zhang2024singularities}]
Let $f:F\rightarrow\Tilde{F}$ be a smooth map between surfaces. 
\begin{itemize}
    \item A point $x\in F$ is called a {\bf singular point} of $f$ if $df_x$ has rank strictly smaller than two.
    \item A critical point $x\in F$ is called a {\bf fold point} if there are local coordinates $x_1,x_2$ centered at $x$ and $y_1,y_2$ centered at $f(x)$ such that $f$ is given by
    \[
    (x_1,x_2)\mapsto(x_1, x_2^2).
    \]
    \item A critical point $x\in F$ is called a {\bf cusp point} if there are local coordinates $x_1,x_2$ centered at $x$ and $y_1,y_2$ centered at $f(x)$ such that $f$ is given by
    \[
    (x_1,x_2)\mapsto(x_1,x_1x_2+x_2^3).
    \]
\end{itemize}
\end{defn}

\begin{defn}[\cite{zhang2024singularities}]\label{def:bifold}
Let $(F_1,\omega_1)$ and $(F_2,\omega_2)$ be two closed symplectic surfaces. Let $g=(g_1,g_2):F\rightarrow F_1\times F_2$ be a smooth Lagrangian immersion. 
\begin{itemize}
    \item A point $x\in F$ is called {\bf bisingular} if $x$ is a critical point for both $g_1$ and $g_2$.
    \item A bisingular point $x\in F$ is called a {\bf bifold point} if $x$ is a fold point for both $g_1$ and $g_2$.
\end{itemize}
\end{defn}

\begin{rmk}
    In \cite{zhang2024singularities}, the author proved that 
    \begin{itemize}
        \item when $g=(g_1,g_2):F\rightarrow F_1\times F_2$ is a Lagrangian immersion, then all singularities of $g_1$ must be bisingular points,
        \item the bisingular points of the Lagrangian immersions $F\looparrowright F_1\times F_2$ are generically bifold points with finitely many cusps (the definition of cusps can be found in \cite{zhang2024singularities}).
    \end{itemize}
\end{rmk}

In the case where the Lagrangian immersion $F\looparrowright F_1\times F_2$ has only embedded bifold points, the following example shows that near the bifold points, there are always figure-eight bubblings when performing strip-shrinking.\\

\begin{ex}\label{exp:3.3}

\begin{figure}[H] 
\centering 
\includegraphics[width=0.6\textwidth]{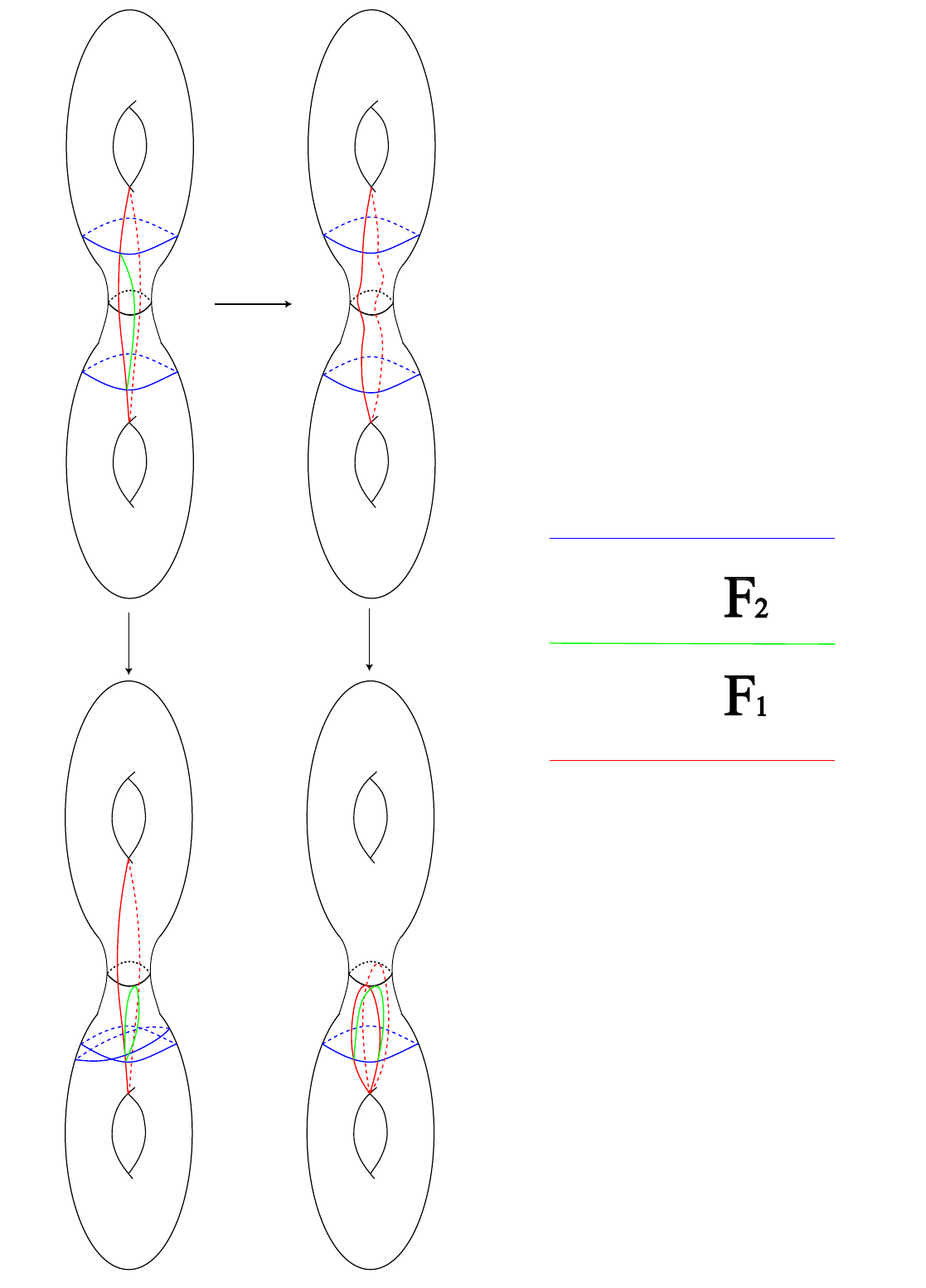}
\caption{The bottom-left and bottom-right genus two surfaces are denoted as $F_1$ and $F_2$, respectively. The red (blue) curve on the bottom-right (left) surface is the result of performing the Lagrangian composition of the red (blue) curve on the bottom-left (right) surface with the genus two surface $F$.}
\label{fig:expwithbubble}
\end{figure}
    The Lagrangian correspondence $(g_1,g_2):F\looparrowright F_1\times F_2$ is illustrated in Figure \ref{fig:expwithbubble}. The horizontal map represents the {0-Dehn twist} (defined as ``good map" in \cite{zhang2024singularities}) along the middle circles \cite{zhang2024singularities}. Both vertical maps are folds along the central black circles: the left vertical map is $g_1$, and $g_2$ is the composition of the horizontal and the right vertical maps. According to the result in \cite{zhang2024construction}, there is a holomorphic quilt whose bottom layer goes to the bigon bounded by the red and the green curves in the $F_1$ and the top layer corresponds the bigon corresponds to the bigon bounded by the green and blue curve in $F_2$.
    
    We prove that there exists a figure-eight bubbling when performing the $F_1$-strip-shrinking (shrinking the bottom layer). If it were not the case, then the family holomorphic quilt $(u^t_1(s,t),u^t_2(s,t+1))$, corresponding to the strip-shrinking process, would have a uniform $C^0$ upper bound for its first derivative as $t\rightarrow1$. By applying the mean-value theorem from mathematical analysis, the inequality below holds:
    \[
    \|u^t_1(s,1)-u^t_1(s,\delta_t)\|_{C^0}\le \|du^t\|_{C^0}(1-\delta_t)\quad \forall s\in\R.
    \]
    Therefore the green and blue boundaries of the bigon in $F_1$ must coincide as $t\rightarrow1$. But this is impossible since the green curve has to intersect the image of the folding circle. Thus, the first derivative of $u^t_1$ blows up as $t\rightarrow1$. If the complex structure on $F_1$ and $F_2$ are assumed to be generic in the sense of \cite{bottman2018gromov}, then shrinking the bottom layer of this quilt produces a figure-eight bubbling between $F$ and $L_2$. The same argument can be applied to the top layer ($F_2$-strip-shrinking) of the quilt in Figure \ref{fig:expwithbubble}, one gets a figure-eight bubbling between $L_1$ and $F$.
\end{ex}

\section{Uniqueness Of Holomorphic Quilts}\label{sec:4}
Suppose that $(F_1,\omega_1)$ and $(F_2,\omega_2)$ are closed symplectic surfaces with compatible complex structures $J_1$ and $J_2$ respectively. Let $L_i$ be Lagrangian immersions into $F_i$ for $i=1,2$ and let $(g_1,g_2):F\looparrowright(F_1\times F_2,\omega_1\times(-\omega_2))$ be a Lagrangian immersion. Assume further that $L_i$ is composable with $F$ for $i=1,2$. Let $u$ be a holomorphic disc connecting two elements in $L_1\times_{F_1} (F\circ L_2)$ with boundary in $L_1$ and $F\circ L_2$. The author proved that if $\im(u)\subset\im(g_1)$, then $u$ can be lifted to a holomorphic quilt of the tuple $(L_1,F,L_2;F_1,F_2)$ \cite{zhang2024construction}. {\bf The goal of this section is to prove the main theorem, namely, that this lift is unique.}

The proof of the main theorem differs from Wehrheim and Woodward’s approach \cite{wehrheim2009floer}, as figure-eight bubbling is always present. For the reader's convenience, the main theorem is restated here.

\begin{usethmcounterof}{thm:main}
    Let $(F_1.\omega_1)$ and $(F_2,\omega_2)$ be two symplectic closed surfaces equipped with compatible almost complex structures. Suppose $L_i\looparrowright F_i$ $\mathrm{(}i=1,2\mathrm{)}$ and $F\looparrowright (F_1\times F_2,\omega_1\times(-\omega_2))$ are Lagrangian immersions such that $L_1\times L_2$ intersects $F$ transversely in $F_1\times F_2$. Furthermore, assume $F$ and $L_2$ are composable. Assume that
    \[
    u(x,y):\R\times[0,1]\rightarrow F_2
    \]
    is a holomorphic map such that     
    \begin{itemize}
        \item $\lim_{x\rightarrow\pm\infty}u(x,y)=x_{\pm}$, where $x_\pm$ are two points in the generalized intersection of $L_1\circ F$ and $L_2$,
        \item $u$ has its boundary in $L_1\circ F$ and $L_2$ $($Definition \ref{def:lagbd}$)$,
        \item the image of $u$ is contained in the image of $F\looparrowright F_1\times F_2$ after projecting to the first factor.
    \end{itemize} 
    Then there is a unique holomorphic map $\tilde u(x,y):\R\times[0,1]\rightarrow F_1\times F_2$ such that
    \begin{itemize}
        \item $\lim_{x\rightarrow\pm\infty}\tilde u(x,y)=x_{\pm}$, where $\tilde x_\pm$ are two points in the generalized intersection of $L_1$ and $F\circ L_2$ corresponding to $x_\pm$ by using the canonical identification in Proposition \ref{prop:sg},
        \item $\tilde u$ has its boundary in $L_1\times L_2$ and $F$,
        \item the image of $\tilde{u}|_{\R\times\{0\}}$ in the $L_2$ component is the same as $\im(u|_{\R\times\{1\}})$ in the sense of set. 
    \end{itemize}
\end{usethmcounterof}

\begin{proof}
    The proof in Wehrheim-Woodward \cite{wehrheim2009floer} relies on the boundedness of the first derivative of the holomorphic quilts during the strip-shrinking process. However, as shown in the previous section, figure-eight bubblings inevitably occur. Therefore, the strategy in \cite{wehrheim2009floer} is no longer applicable. The main idea here is to remove a neighborhood around the bubbling points, as well as a thin rectangle connecting this neighborhood to the boundary (see the left picture of Figure \ref{fig:keyholwremove}), and then apply the Riemann mapping theorem.

    Suppose $(u_1,u_2)$ and $(w_1,w_2)$ are two smoothly isotopic holomorphic quilts connecting the same pair of generalized intersections $((x_+,y_+),(x_-,y_-))$ of $F$ and $L_1\times L_2$. Because both $u_2$ and $w_2$ are holomorphic discs connecting $y_\pm$, $u_2$ and $w_2$ are smoothly isotopic. By applying the $F_1$-strip-shrinking to $(u_1,u_2)$ and $(w_1,w_2)$, we obtain two families of holomorphic quilts $(u_1^t,u^t_2)$ and $(w^t_1,w^t_2)$ starting at $(u_1,u_2)$ and $(w_1,w_2)$ for $t\in[0,1)$. To be more precise,
    \begin{itemize}
        \item the smooth maps
        \[
        u_1^t,\ w_1^t:\R\times[\delta^t,1]\rightarrow F_1;\ \ u_2^t,\ w_2^t:\R\times[1,2]\rightarrow F_2
        \]
        are holomorphic for $t\in[0,1)$, where $\delta^t$ is an increasing function with respect to $t$ such that $\delta^0=0$ and $\delta^1=1$,
        \item the maps $u_1^t(s,\delta^t)$, $u_2^t(s,2)$, and $(u_1^t,u_2^t)(s,1)$ can be lifted to $L_1$, $L_2$, and $F$ respectively for $t\in[0,1)$,
        \item the maps $w_1^t(s,\delta^t)$, $w_2^t(s,2)$, and $(w_1^t,w_2^t)(s,1)$ can be lifted to $L_1$, $L_2$, and $F$ respectively for $t\in[0,1)$.
    \end{itemize}    
    Without loss of generality, assume $u^t_2(0,2)=w^t_2(0,2)$.

    Let $\{b_j\}_{j=1}^l$ denote the set of points where the derivatives of $(u_1,u_2)$ and $(w_1,w_2)$ blow up during the strip-shrinking process. Since the complex structures on the symplectic manifolds are assumed to be generic, there are only figure-eight bubblings in the strip-shrinking process \cite{wehrheim2009floer}. According to Gromov compactness, there are small discs $B_{\epsilon^t_j}(s_j^t,o_j^t)\subset \R\times[1,2]$ around $(s_j^t,o_j^t)$ with radius $\epsilon^t_j$ such that
    \begin{itemize}
        \item $\lim_{t\rightarrow 1}u_2^t(s_j^t,o_j^t)=\lim_{t\rightarrow 1}w_2^t(s_j^t,o_j^t)=b_j$,
        \item $u_2^t$ and $w_2^t$ converges uniformly outside $B_{\epsilon_j^t}(s_j^t,o^t_j)$ as $t\rightarrow1$.
    \end{itemize}
    \begin{figure}[H]
    \centering 
    \includegraphics[width=0.8\textwidth]{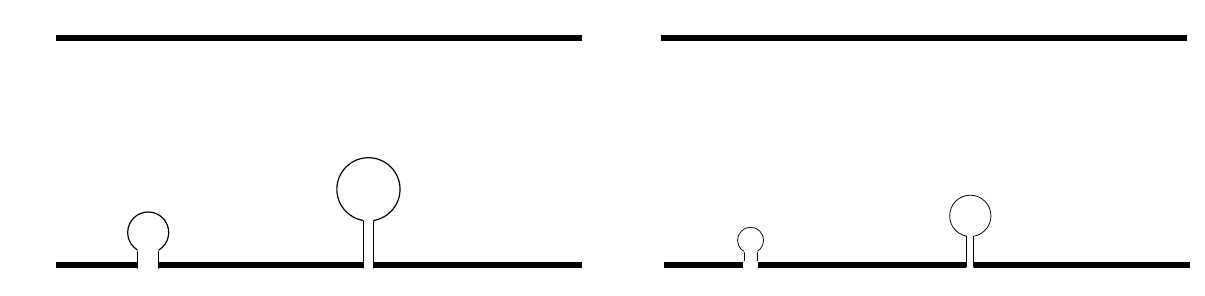}
    \caption{As $t$ tends to 1, the keyholes become smaller and smaller as shown from left to the right in the above picture.}
    \label{fig:keyholwremove}
    \end{figure}
    By removing $B_{\epsilon^t_j}(s_j^t,o_j^t)$ as well as a thin rectangle connecting $B_{\epsilon^t_j}(s_j^t,o_j^t)$ and the boundary $\R\times\{1\}$ from $\R\times[1,2]$, we obtain a family of simply connected sets $S^t$ (Figure \ref{fig:keyholwremove}). The width of these thin rectangles are chosen to converge to 0 as $t\rightarrow1$. Then $u_2^t|_{S^t}$ and $w_2^t|_{S^t}$ are defined on simply connected spaces. By applying the Riemann mapping theorem, we may assume that the domains and the target of $u^t_2|_{S^t}$ and $w_2^t|_{S^t}$ are the unit disc in $\C$, with the following conditions:
    \begin{itemize}
        \item $u^t_2(1,0)=w^t_2(1,0)=(1,0)$, 
        \item $u^t_2(-1,0)=w^t_2(-1,0)=(-1,0)$,
        \item $u^t_2(0,-1)=w^t_2(0,-1)$ (this comes from the assumption $u^t_2(0,2)=w^t_2(0,2)$).
    \end{itemize}
    Since $u_2^t|_{S^t}$ and $w_2^t|_{S^t}$ are holomorphic, then $u_2^t|_{S^t}=w_2^t|_{S^t}$ according to complex analysis. Therefore, $u_2^t=w_2^t$ by the unique continuation theorem in complex analysis. Because $(u_1^t,u^t_2)|_{\R\times\{1\}}$ can be lifted to $F$, $u_1^t|_{\R\times\{1\}}=w_1^t|_{\R\times\{1\}}$ is uniquely determined. By identifying the domains and the targets of $u^t_1$ and $w_1^t$ with the unit disc in $\C$ using the Riemann mapping theorem and combining the fact that $u_1^t$ and $w_1^t$ are uniquely determined by the value restricted on ${\R\times\{1\}}$. As a result, $u_1^t=w_1^t$. This finishes the proof of the theorem.
\end{proof}

\section{Examples of the Quilted Lagrangian Floer Theory}\label{sec:5}
In this section, some examples of the quilted Lagrangian Floer theory are calculated using the techniques that we have developed so far. The final example confirms the conjecture in Section 10.4 of Cazassus-
Herald-Kirk-Kotelskiy \cite{cazassus2020correspondence}.

\begin{ex}
    This example is contiues Example \ref{exp:3.3} (see Figure \ref{fig:expwithbubble}). As discussed in Example \ref{exp:3.3}, there is a holomorphic quilt as depicted in Figure \ref{fig:expwithbubble}. According to Theorem \ref{thm:main}, there is only one holomorphic quilt that is topologically homotopic to the one shown in the figure. Notice that there is a holomorphic quilt connecting the two generators in the backside of Figure \ref{fig:expwithbubble}. Therefore, the quilted Lagrangian Floer homology corresponding to Figure \ref{fig:expwithbubble} is trivial.
\end{ex}

    \begin{figure}[H] 
    \centering 
    \includegraphics[width=0.4\textwidth]{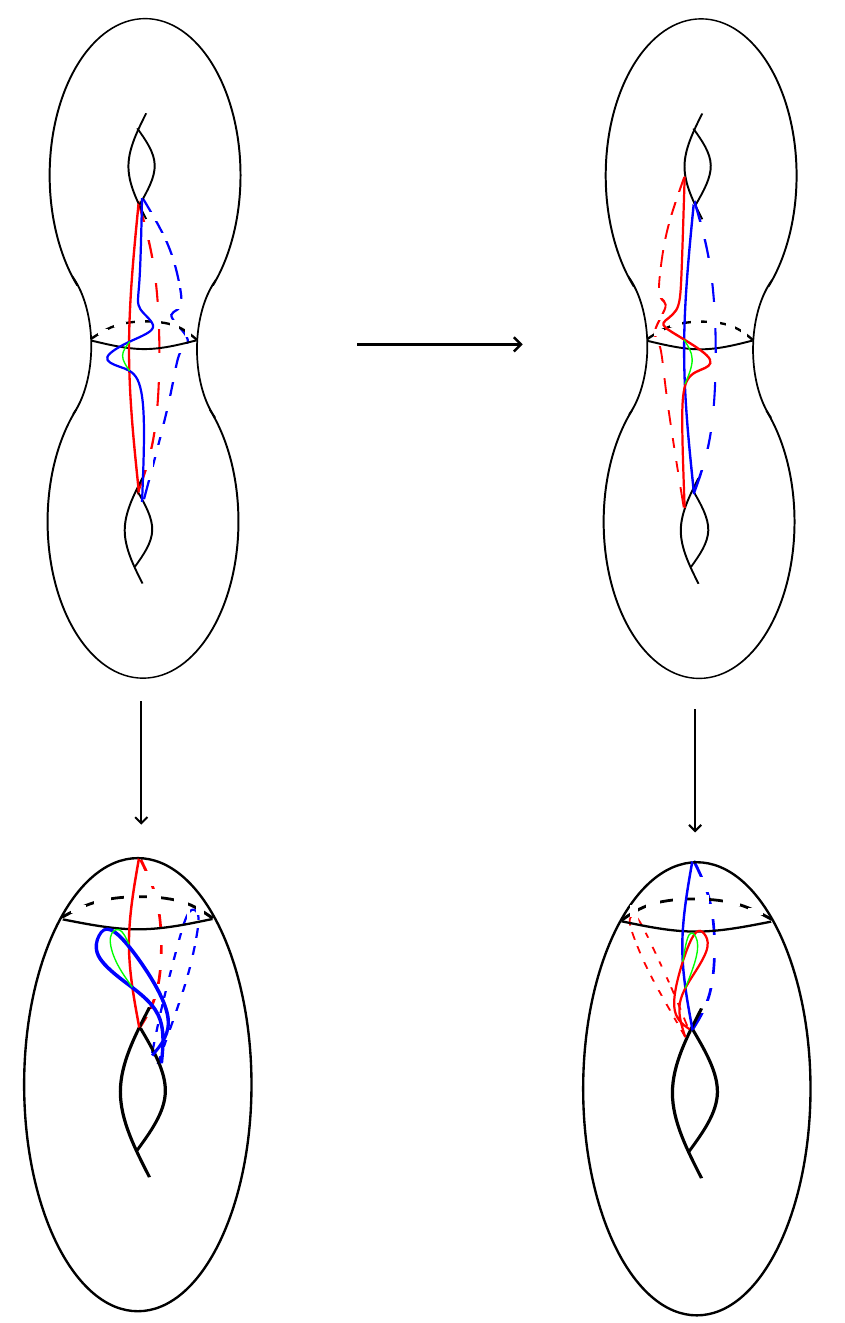}
    \caption{The bottom-left and bottom-right genus-two surfaces are $F_1$ and $F_2$ respectively. The red (blue) curve on the bottom-right (left) surface is obtained by performing the Lagrangian composition of the red (blue) curve on the bottom-left (right) surface with the genus-two surface $F$.}
    \label{fig:Expquilted1}
    \end{figure}

\begin{ex}
    The Lagrangian correspondence $(g_1,g_2):F\looparrowright F_1\times F_2$ is illustrated in Figure \ref{fig:Expquilted1}. The horizontal map represents the {0-Dehn twist} along the middle circles \cite{zhang2024singularities}. Both vertical maps are folds along the central black circles: the left vertical map is $g_1$, and $g_2$ is the composition of the horizontal and the right vertical maps. According to the result in \cite{zhang2024construction} and Theorem \ref{thm:main}, there is a unique holomorphic quilt (the right part of Figure \ref{fig:Expquilted1}) whose bottom layer maps to the bigon bounded by the red and green curves in $F_1$ and the top layer maps to the bigon bounded by the green and blue curve in $F_2$. As a result, the quilted Lagrangian Floer homology is trivial.
\end{ex}

\begin{ex}
    The Lagrangian correspondence $(g_1,g_2):F\looparrowright F_1\times F_2$ is illustrated in Figure \ref{fig:Expquilted2}. The horizontal map represents the {0-Dehn twist} (it is defined as ``good map" in \cite{zhang2024singularities}) along the two circles \cite{zhang2024singularities} on the bubble. Both vertical maps are by pressing the bubble to the upper half part of the genus two surface. According to the result in \cite{zhang2024construction} and Theorem \ref{thm:main}, there is a unique holomorphic quilt (the right side of Figure \ref{fig:Expquilted2}) whose bottom layer maps to the bigon bounded by the blue and the green curves in the $F_1$ and the top layer maps to the bigon bounded by the green and red curve in $F_2$. As a result, the quilted Lagrangian Floer homology is trivial.

    \begin{figure}[H] 
    \centering 
    \includegraphics[width=1\textwidth]{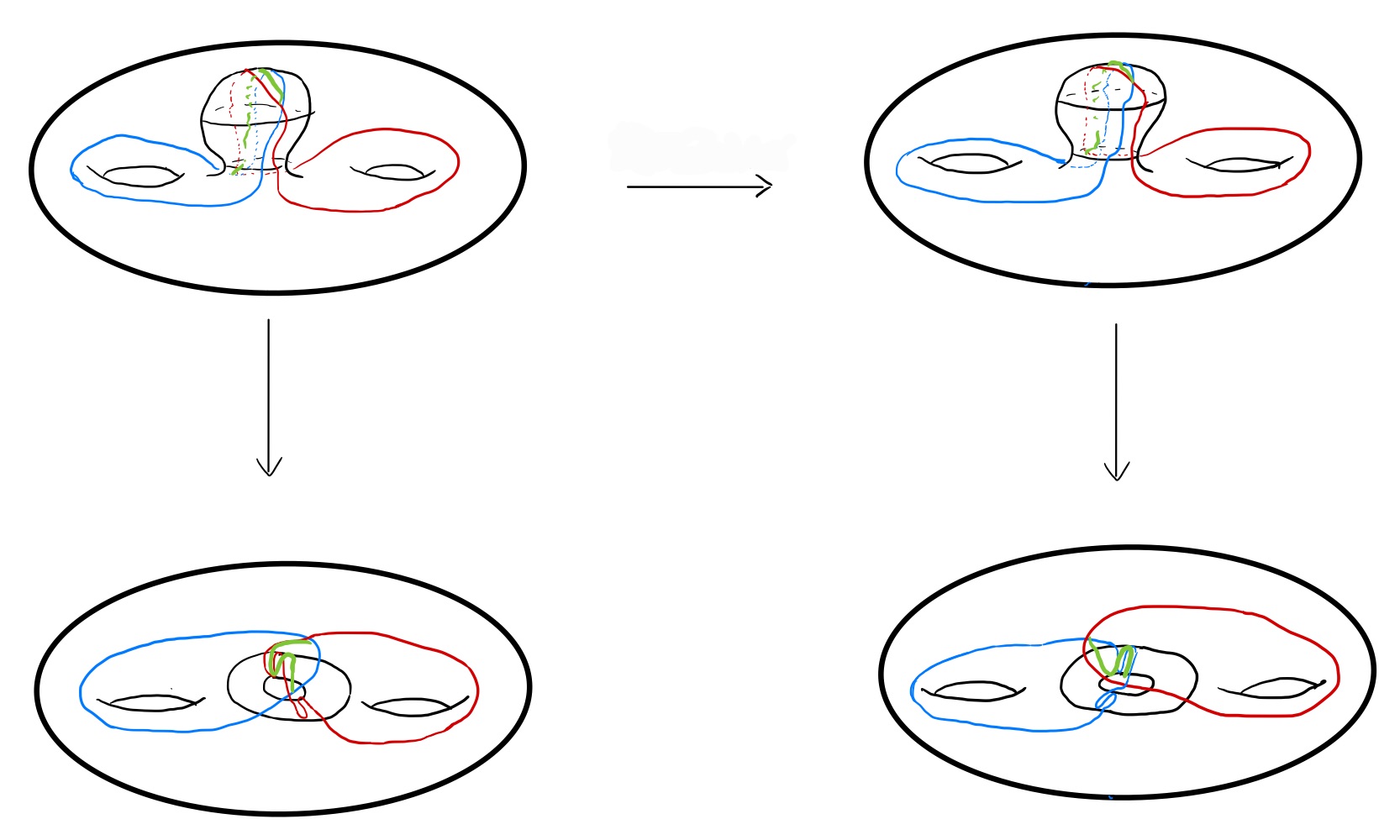}
    \caption{The bottom-left and bottom-right genus-two surfaces are $F_1$ and $F_2$ respectively. The blue (red) curve on the bottom-right (left) surface is obtained by performing the Lagrangian composition of the blue (red) curve on the bottom-left (right) surface with the genus-two surface $F$.}
    \label{fig:Expquilted2}
    \end{figure}
\end{ex}

\begin{ex}
    In Section 10.4 of Cazassus-Herald-Kirk-Kotelskiy \cite{cazassus2020correspondence}, the authors conjecture that there is a unique holomorphic quilt corresponding to the immersed smooth quilt depicted in Figure 22. This assertion holds true, by the same argument used in the previous examples. Therefore, a figure-eight bubbling necessarily appears when performing strip-shrinking in \cite{cazassus2020correspondence}.
\end{ex}

\bibliographystyle{plain}
\bibliography{references}

\begin{thebibliography}{10}

\bibitem{abouzaid2008fukaya}
Mohammed Abouzaid.
\newblock On the {F}ukaya categories of higher genus surfaces.
\newblock {\em Advances in Mathematics}, 217(3):1192--1235, 2008.

\bibitem{bottman2018gromov}
Nathaniel Bottman and Katrin Wehrheim.
\newblock Gromov compactness for squiggly strip shrinking in pseudoholomorphic quilts.
\newblock {\em Selecta Mathematica}, 24(4):3381--3443, 2018.

\bibitem{cazassus2020correspondence}
Guillem Cazassus, Christopher~M Herald, Paul Kirk, and Artem Kotelskiy.
\newblock The correspondence induced on the pillowcase by the earring tangle.
\newblock {\em arXiv preprint arXiv:2010.04320}, 2020.

\bibitem{floer1988morse}
Andreas Floer.
\newblock Morse theory for {L}agrangian intersections.
\newblock {\em Journal of differential geometry}, 28(3):513--547, 1988.

\bibitem{fukaya2017unobstructed}
Kenji Fukaya.
\newblock Unobstructed immersed {L}agrangian correspondence and filtered a infinity functor.
\newblock {\em arXiv preprint arXiv:1706.02131}, 2017.

\bibitem{goldman1984symplectic}
William~M Goldman.
\newblock The symplectic nature of fundamental groups of surfaces.
\newblock {\em Advances in Mathematics}, 54(2):200--225, 1984.

\bibitem{herald2024endomorphism}
Christopher~M Herald and Paul Kirk.
\newblock An endomorphism on immersed curves in the pillowcase.
\newblock {\em arXiv preprint arXiv:2407.11247}, 2024.

\bibitem{mcduff2017introduction}
Dusa McDuff and Dietmar Salamon.
\newblock {\em Introduction to symplectic topology}, volume~27.
\newblock Oxford University Press, 2017.

\bibitem{smith2023perturbed}
Kai Smith.
\newblock {\em Perturbed Traceless SU (2) Character Varieties of Tangle Sums}.
\newblock Indiana University, 2023.

\bibitem{wehrheim2009floer}
Katrin Wehrheim and Chris~T Woodward.
\newblock Floer cohomology and geometric composition of {L}agrangian correspondences.
\newblock {\em arXiv preprint arXiv:0905.1368}, 2009.

\bibitem{wehrheim2010functoriality}
Katrin Wehrheim and Chris~T Woodward.
\newblock Functoriality for {L}agrangian correspondences in {F}loer theory.
\newblock {\em Quantum topology}, 1(2):129--170, 2010.

\bibitem{wehrheim2010quilted}
Katrin Wehrheim and Chris~T Woodward.
\newblock Quilted {F}loer cohomology.
\newblock {\em Geometry \& Topology}, 14(2):833--902, 2010.

\bibitem{zhang2024construction}
Zuyi Zhang.
\newblock Construction of holomorphic quilts in {C}artesian product of closed surfaces.
\newblock {\em arXiv preprint arXiv:2409.19744}, 2024.

\bibitem{zhang2024singularities}
Zuyi Zhang.
\newblock {\em Singularities of Lagrangian Immersions and Applications in Lagrangian Floer Theory}.
\newblock PhD thesis, Indiana University, 2024.

\end{thebibliography}

\end{document}